\documentclass{amsart}
\newtheorem{theorem}{Theorem}[section]
\newtheorem{lemma}[theorem]{Lemma}

\title[local conformal rigidity in codimension  $\leq$ 5]
 {local conformal rigidity in codimension  $\leq$ 5}

\author{Sergio L. Silva}
\dedicatory{Universidade do Estado do Rio de Janeiro,
Instituto de Matem\'atica e Estat\'\i stica, Departamento de Estrutura Matemática, 
Rio de Janeiro, 20550-013,
Brazil, E-mail: sergiol@ime.uerj.br}
\keywords{local, conformal, rigidity}
\subjclass[2010]{Primary 53A30, 53C40; Secondary }

\begin{document}
\maketitle

\begin{abstract}
In this paper, for an immersion $f$ of an $n$-dimensional Riemannian manifold $M$ into $(n+d)$-Euclidean space we give a sufficient condition on $f$ so that, in case $d\leq 5$, any immersion $g$ of $M$ into $(n+d+1)$-Euclidean space that induces on $M$ a  metric that is conformal to the metric induced by $f$ is locally obtained, in a dense subset of $M$, by a composition of $f$ and a conformal immersion from an open subset of $(n+d)$-Euclidean space into an open subset of $(n+d+1)$-Euclidean space.  Our result extends a theorem for hypersurfaces due to M. Dajczer and E. Vergasta. The restriction on the codimension is related to a basic lemma in the theory of rigidity obtained by M. do Carmo and M. Dajczer.
\end{abstract}

\section{Introduction}

Let $M^n$ be an $n$-dimensional differentiable manifold and let $f\colon\;M^n\to \Bbb R^m$ and $g\colon\;M^n\to\Bbb R^\ell$ be two immersions into Euclidean spaces. We say that $\,g\,$ is conformal to $\,f\,$ when the metric induced on $\,M^{n}\,$ by $\,f\,$ and $\,g\,$ are conformal. That $f$ is {\it conformally rigid} means that for any other conformal immersion $h\colon\;M^n\to\Bbb R^m$, conformal to $f$, there exists a conformal diffeomorphism $\,\Upsilon\,$ from an open subset of $\Bbb R^m$ to an open subset of $\Bbb R^m$ such that $\,h=\Upsilon \circ f$.

In \cite{C-D}, do Carmo and Dajczer introduced a conformal invariant for an immersion $f\colon\;M^n\to\tilde{M}^{n+d}$ into a Riemannian manifold $\tilde{M}^{n+d}$, namely, {\it the conformal $s$-nullity} $\nu_{s}^{c}(p)$, $p\in M^n$, $1\leq s\leq d$(see Section 2 for definitions), and proved that is conformally rigid an immersion $f\colon\;M^n\to \Bbb R^{n+d}$ that satisfies $\,d\leq 4\,$, $\,n\geq 2d+3$ and $\,\nu_{s}^{c}\leq n-2s-1\,$for $\,1\leq s \leq d$. Their result generalizes a result for hypersurfaces due to E. Cartan (\cite{Ca}). It was observed in Corollary 1.1 of \cite{Si} that the do Carmo-Dajczer's conformal rigidity result also holds for $d=5$. The restriction on the codimension is due to the following basic result in \cite{C-D}

\begin{theorem} Let $\sigma\colon\; V_{1}\times V_{1}\to W^{(r,r)}$ be a nonzero flat symmetric bilinear form. Assume $\,r\leq 5\,$ and $\mbox{dim}\,N(\sigma)< \mbox{dim}\,V_{1}-2r$. Then
$\,S(\sigma)\,$ is degenerate.
\end{theorem}

In the
 paper(see \cite{D-F}), M. Dajczer and L. Florit proved that the above theorem can not be improved.

Denoting by $N_{cf}^{n+d}$ a conformally flat Riemannian manifold of dimension $n+d$, namely an $(n+d)$-dimensional Riemannian manifold $N^{n+d}$ which is locally conformally diffeomorphic to an open subset of Euclidean space $\Bbb R^{n+d}$ with the canonical metric, Dajczer and Vergasta(\cite{D-V}) proved that if $f\colon\;M^n\to N_{cf}^{n+1}$ satisfies $n\geq 6$ and $\nu_{1}^{c}\leq n-4$ along $M^n$,  then any immersion $g\colon\;M^n\to N_{cf}^{n+2}$ conformal to $f$ is locally a composition in an open dense subset $\mathcal{U}$ of $M$. This result, still in the context of hypersurfaces, was extended by Dajczer-Tojeiro in \cite{D-T} for $g\colon\;M^n\to N_{cf}^{n+p}$, $2\leq p\leq n-4$,  assuming that $\nu_{1}^{c}\leq n-p-2$ and, if $p\geq 6$, further that $M^n$ does not contain an open $n-p+2$-conformally ruled subset for both $f$ and $g$. In this paper, we extend for codimension $\leq 5$ the result of Dajczer-Vergasta mentioned above in the following theorems.

\begin{theorem}Let $f\;\colon M^n\to \Bbb R^{n+d}$ be an immersion with $d\leq 5$ and $\,n > 2d+3$.  Assume that everywhere $\,\nu_s^c\leq n-2s-2,$ $1\leq s \leq d.$ If $g\colon\;M^n\to\Bbb R^{n+d+1}$ is an immersion conformal to $f$, then there exists an open dense subset $\mathcal{U}$ of $M^n$ such that $g$ restricted to $\mathcal{U}$ is locally a composition $g=\Gamma\circ f$  for some local conformal immersion $\Gamma$ from an open subset of $\,\Bbb R^{n+d}$ into an open subset of $\,\Bbb R^{n+d+1}$.\end{theorem}

Observing that Theorem 1.2 is local one, the following result is an immediate consequence of Theorem 1.2.

\begin{theorem} Let $f\;\colon M^n\to N_{cf}^{n+d}$ be an immersion with $d\leq 5$ and $\,n > 2d+3$.  Assume that everywhere $\,\nu_{s}^{c}\leq n-2s-2,$ $1\leq s \leq d.$ If $g\colon\;M^n\to N_{cf}^{n+d+1}$ is an immersion conformal to $f$, then there exists an open dense subset $\mathcal{U}$ of $M^n$ such that $g$ restricted to $\mathcal{U}$ is locally a composition $g=\Gamma\circ f$  for some local conformal immersion $\Gamma$ from an open subset of $N_{cf}^{n+d}$ into an open subset of $N_{cf}^{n+d+1}$.\end{theorem}

\section{Proof of Theorem 1.2}

For a symmetric bilinear form $\beta\colon\;V\times V\to W$ we denote by $\,S(\beta )\,$ the subspace of $\,W\,$ given by 
$$ \,S(\beta ) =\,\mbox{span}\,\{\,\beta (X,Y)\,:\,X,Y \in V\,\}, $$ 
and by $\,N(\beta)\,$ the nullity space of
$\,\beta\,$ defined as 
$$ \,N(\beta) =\, \{\,n \in V \,:\, \beta (X,n)=0,\,\,\forall\, X \in V\,\}. $$

For an immersion $f\colon\;M^n\to\tilde{M}^{n+d}$ into a Riemannian manifold we denote by $\,\alpha^f\; \colon TM\times TM \to T^{\perp}M$ its vector valued second fundamental form and by $T_{f(p)}^{\perp}M$ the normal space of $f$ at $p\in M^n$. Given an $\,s$-dimensional subspace $\,U^{s} \subseteq T_{f(p)}^\perp M ,$ $1\leq s \leq d,$ consider the symmetric bilinear form 
$$\alpha_{U^{s}}^{f}\colon\;T_pM\times T_pM\to U^s$$ 
defined as $\,\alpha_{U^{s}}^{f}=P \circ \alpha^f,$ where $\,P\,$ denotes the orthogonal projection of $T_{f(p)}^\perp $ onto $\,U^{s}$. The {\it conformal} {\it $\,s$-nullity} $\,\nu_s^c(p)\,$ of $\,f\,$ at $\,p$, $1\leq s\leq d$, is the integer 
$$
\nu_{s}^{c}(p)=\max_{U^{s} \subseteq T_{f(p)}^\perp ,\,\eta \in U^{s}}\left\{\mbox{dim}\,N\left(\alpha_{U^{s}}^{f}-\langle\,\,,\,\,\rangle \eta\right)\right \} 
$$

The Lorentz space $\,I\!\!L^{k},\,k\geq 2,$ is the Euclidean space $\,\Bbb R^{k}\,$ endowed with the metric $\,\langle\,\,,\,\,\rangle\,$ defined by 
$$ 
\langle X,X \rangle = -x_{1}^{2}+x_{2}^{2}+\cdots + x_{k}^{2} 
$$ 
for $\,X=(x_{1},x_{2}, \cdots , x_{k}).$  
The {\it light cone} $\,\mathcal{V}^{k-1}\,$ is the degenerate totally umbilical hypersurface of $\,I\!\!L^{k}\,$ given by 
$$ 
\mathcal{V}^{k-1}=\{\,X \in 
I\!\!L^{k}\,:\,\langle X,X \rangle = 0\;,\; X \neq 0\,\}. 
$$ 
For $k\geq 3$ and $\,\zeta \in \mathcal{V}^{k-1}\,$ consider the hyperplane
$$ 
H_{\zeta}=\left\{\,X \in I\!\!L^{k}\,:\,\langle X,\zeta \rangle = 1\right\} 
$$ 
and the $\,(k-2)$-dimensional simply connected embedded submanifold $\,H_{\zeta} \cap \mathcal{V}^{k-1}\subset I\!\!L^{k}.$ Note that $H_\zeta$ intersects only one of the two connected components of $\mathcal{V}^{k-1}.$ More precisely, $\,H_{\zeta}$ is in the Euclidean sense parallel to $\zeta$ and it does not pass through the origin. Given $p\in H_{\zeta} \cap \mathcal{V}^{k-1}$ the normal space of this intersection in $\,I\!\!L^{k}\,$ is the Lorentzian plane $\,I\!\!L^{2}\,$ generated by $\,p\,$ and $\,\zeta.$ Consequently, the metric induced by $\,I\!\!L^{k}\,$ on $\,H_{\zeta} \cap \mathcal{V}^{k-1}\,$  is Riemannian. Its second fundamental form is given by
$$ \alpha =-\langle\,\,,\,\,\rangle \zeta\,. $$ 
The Gauss equation for the inclusion $\,H_{\zeta} \cap \mathcal{V}^{k-1}\subset I\!\!L^{k}$ shows that $\,H_{\zeta} \cap \mathcal{V}^{k-1}$ is  flat and, consequently, is the image of an isometric embedding
$\,J_\zeta\colon\;\Bbb R^{k-2}\to H_\zeta \cap \mathcal{V}^{k-1}$. 

Now suppose that $M^n$ is a Riemannian manifold and $h\colon\;M^n\to\Bbb R^{k-2}$ is a conformal immersion with conformal factor $\phi_h$, that is, $\langle {h}_{\ast}X,{h}_{\ast}Y \rangle =\phi_{h}^{2} \langle X,Y \rangle $ over $M^n,$ where $X,\,Y$ are vectors tangent to $M^n$ and $\phi_h$ is a positive differentiable real function on $M^n$. We associate to $\,h\,$ the isometric immersion $H\colon\;M^n\to \mathcal{V}^{k-1}\subset I\!\!L^k\,$ by setting 
$$
H=\frac{1}{\phi_h}J_\zeta\circ h 
$$ 
for a chosen $\,\zeta \in \mathcal{V}^{k-1}.$ 

Consider $M^n$ with the metric induced by $f$, a fixed $\,\zeta$ in $\mathcal{V}^{n+d+1}$ and an isometric embedding $J_\zeta\colon\Bbb R^{n+d}\to H_\zeta \cap \mathcal{V}^{n+d+1}$. The isometric immersion $F\colon\;M^n\to \mathcal{V}^{n+d+1}\subset I\!\!L^{n+d+2}\,$ associated to $f$ is given by
$$
F=J_\zeta \circ f.
$$
Its second fundamental form in $I\!\!L^{n+d+2}$ is the symmetric bilinear form
\begin{equation}
\alpha^F=-\langle\,,\,\rangle\zeta +\alpha^f.\label{1}
\end{equation}
Here, we are identifying the second fundamental form $\alpha^f$ of $f$ in $\Bbb R^{n+d}$ with the symmetric bilinear form $\left(J_\zeta\right)_*\alpha^f .$

Now let $g\colon\;M^n\to \Bbb R^{n+d+1}$ be an immersion conformal to $f.$ Consider the isometric immersion $G\colon\;M^n\to \mathcal{V}^{n+d+2}\subset I\!\!L^{n+d+3}$ given by
\begin{equation}
G=\frac{1}{\phi_{g}}J_{\overline{\zeta}}\circ g \label{2}
\end{equation} 
for an fixed $\,\overline{\zeta} $ in $\mathcal{V}^{n+d+2},$ where $J_{\overline{\zeta}}\colon\;\Bbb R^{n+d+1}\to H_{\overline{\zeta}} \cap \mathcal{V}^{n+d+2}$ is an isometric embedding. Taking the derivative of $\,\langle G , G\rangle =0$, we see that the null vector field $\,G\,$ is normal to the immersion $\,G.$ The normal field $\,G\,$ also satisfies $\,A_{G}^{G}=-I.$ The normal bundle of $\,G\,$ is given by the orthogonal direct sum 
$$
T_{G}^{\perp}M=T_{g}^{\perp}M \oplus I\!\!L^{2} 
$$ 
where $\,T_g^{\perp}M\,$ is identified with $\,(J_{\overline{\zeta}})_{\ast}T_g^{\perp}M\,$ and $\,I\!\!L^{2}\,$ is a Lorentzian plane bundle that contains $\,G$. There exists an unique orthogonal basis $\,\{ \xi ,\eta \}\,$ of $\,I\!\!L^{2},$ satisfying $\,| \xi |^{2}=-1,$ $\,| \eta |^{2}=1$ and 
$$
G = \xi + \eta.
$$ 
Writing $\,\alpha^{G}\,$ in terms of this orthogonal frame we obtain  
$$ 
\alpha^{G}=- \left\langle \alpha^{G} , \xi\right\rangle \xi + \left\langle \alpha^{G} , \eta \right\rangle \eta + \left(\alpha^G\right)^{\ast} 
$$
where $\,\left(\alpha^G\right)^{\ast}=(1/\phi_{g})(J_{\overline{\zeta}})_{\ast}\alpha^{g}\,$ is the $\,T_{g}^{\perp}M$-component of $\,\alpha^{G}.$

Given an $\,m$-dimensional real vector space $\,W\,$ endowed with a nondegenerate inner product $\,\langle\,\,,\,\,\rangle\,$ of {\it index\/} $\,r,$ that is, the maximal dimension of a subspace of $\,W\,$ where
$\,\langle\,\,,\,\,\rangle\,$ is negative definite, we say that $\,W\,$ is of type $\,(r,q)\,$ and we write $\,W^{(r,q)}\,$ with $\,q=m-r.$  

At $\,p \in M^{n},$ let 
$$ 
W_p=T_{f(p)}^\perp M \oplus \mbox{span}\{\xi (p) \} \oplus \mbox{span}\{\eta (p)\}\oplus T_{g(p)}^\perp M 
$$ 
be endowed with the natural metric of type $\,(d+1,d+2)\,$ which is negative definite on $\,T_{f(p)}^\perp M \oplus \mbox{span}\{\xi (p) \}.$ We also define the symmetric tensor  $\beta\colon\;TM\times TM\to W$ setting $\,\beta=\alpha^{f} + \alpha^{G}\,$, that is, 
$$ 
\beta = \alpha^{f} - \left\langle \alpha^{G} , \xi \right\rangle \xi + \left\langle \alpha^{G} , \eta \right\rangle \eta + \left(\alpha^G\right)^{\ast}.  
$$ 
The Gauss equations for $\,f\,$ and $\,G\,$ imply that $\,\beta\,$ is {\it flat}, that is, 
$$ 
\langle \beta (X,Y), \beta (Z,U) \rangle = \langle \beta (X,U), \beta (Y,Z) \rangle ,\;\;\;\;\forall\, X,Y,Z,U \in TM.  
$$ 
Observe also that $\,\beta (X,X) \neq 0\,$ for all $\,X \neq 0,$ because $\,A_{\xi + \eta}^{G}=-I.$

With the aim of constructing locally the conformal immersion $\Gamma$ we construct locally an isometric immersion $\mathcal{T}$ of a neighborhood of $L^{n+d+2}$ into a neighborhood of $L^{n+d+3}$ such that $\,G=\mathcal{T}\circ F.$ This $\,\mathcal{T}\,$ induces a conformal immersion $\Gamma$ from an open subset of $\Bbb R^{n+d}$ into an open subset of $\Bbb R^{n+d+1}$ defined by 
$$ 
J_{\overline{\zeta}}\circ \Gamma =\frac{\mathcal{T}\circ J_{\zeta}}{\langle \mathcal{T}\circ J_{\zeta},\overline{\zeta}\rangle} 
$$ 
which satisfies $\,g=\Gamma \circ f$. Now we will construct locally the isometric immersion $\mathcal{T}.$

\begin{lemma} Given $p\in M^n,$ there exist an orthonormal basis $\zeta_1,\ldots ,\zeta_d$ of $\,T_{f(p)}^\perp M$ and a basis $G,\,\mu_{o},\ldots,\mu_{d}$, $\mu_{d+1}$ of $\,T_{G(p)}^{\perp}M,$ with 
$\,\langle G,\mu_{o}\rangle = 1,\,\langle \mu_{o},\mu_{o}\rangle =0,$ $\,\langle \mu_{o},\mu_{j}\rangle =0=\langle G,\mu_{j}\rangle ,$ $\,\langle \mu_{i},\mu_{j}\rangle =\delta_{ij},$ $1\leq i,\,j \leq d+1,$ such that
\begin{equation} 
\alpha^{G}=-\langle\,\,,\,\,\rangle \mu_{o} +\sum_{j=1}^{d}\left\langle \alpha^{f} , \zeta_{j}\right\rangle\mu_{j}+ \left\langle A_{\mu_{d+1}}^G(.)\,,\,.\, \right\rangle \mu_{d+1}, \label{3}
\end{equation}               
with $\mbox{rank}\,A_{\mu_{d+1}}^G\leq 1.$ Here, $\left\langle \alpha^{f} , \zeta_{j}\right\rangle$ is the inner product induced in $\,T_{f(p)}^{\perp}M$ by $\Bbb R^{n+d}$.
\end{lemma}

With respect to flat symmetric bilinear forms we need the following from \cite{Da}:

\vspace{0.2cm}

For $p\in M^n,$ set $\,V\colon =T_{p}M\,$ and for each $\,X \in\,V\,$ define the linear map 
$$ \beta (X)\; \colon\; V \to W $$ 
by setting $\,\beta (X)(v)=\beta (X,v)\,$ for all $\,v \in V.$ The kernel and image of $\,\beta (X)\,$ are denoted by $\,\mbox{ker}\,\beta (X)\,$ and $\,\beta (X,V),$ respectively. We say that $\,X\,$ is a {\it regular element} of $\,\beta\,$ if
$$\mbox{dim}\,\beta (X,V)=\max_{Z\in V}\mbox{dim}\,\beta (Z,V).$$ 
The set of regular elements of $\,\beta\,$ is denoted by $\,RE(\beta ).$ For each $X\in V,$ set $U(X)=\beta (X,V)\cap \beta (X,V)^{\perp}$ and define
$$RE^{*}(\beta )=\{\,Y\in RE(\beta )\,:\,\mbox{dim}\,U(Y)=d_{0}\,\}$$
where $d_{0}=\min\{\,\mbox{dim}\,U(Y)\,:\,Y\in RE(\beta )\}.$

\begin{lemma} The set $\,RE^{*}(\beta )\,$ is open and dense in $\,V\,$ and 
$$
\beta (\mbox{ker}\,\beta (X),V) \subset U(X),\;\;\;\;\forall\, X \in RE(\beta ).
$$
\end{lemma}

\noindent{Recall} that a vector subspace $\,L\,$ of $\,W\,$ is said to be  {\it degenerate} if $\,L\cap L^{\perp}\neq \{0\}\,$ and {\it isotropic} when $\,\langle L,L\rangle =0.$ We also have that 
\begin{equation}
\mbox{dim}\,L+\mbox{dim}\,L^{\perp}=\mbox{dim}\,W \;\;\;\;\;\mbox{and}\;\;\;\;\; L^{\perp \perp}=L.\label{4} 
\end{equation}  
Notice that $L$ is an isotropic subspace of $W$ if and only if $L\subset L^{\perp}.$ In this case, it follows easily from (\ref{4}) that $L$ has dimension at most $d+1.$ Recall that $W_p$ is a $(2d+3)$-dimensional vector space. Since $U(X)$ is isotropic by definition, we have that $d_{0}\leq d+1.$

Before proving Lemma 2.1 we need several lemmas:

\begin{lemma} In an arbitrary $p\in M^n$ it holds that $\,\mbox{dim}\,S(\beta )\cap S(\beta )^{\perp}= d+1.$\end{lemma}

\begin{proof} First we prove the following assertion

\vspace{0.2cm}

\noindent{\bf Assertion.} {\em There exist an orthogonal decomposition
$$
\,W_p=W_{1}^{(\ell ,\ell)}\oplus W_{2}^{(d-\ell+1\,,\,d-\ell+2)},\;\;\ell \geq 1\;,
$$ 
and symmetric bilinear forms $\omega_j\colon\;V\times V\to W_j$, $j=1,\,2,$ satisfying 
$$ 
\beta =\omega_1 \oplus\omega_2 
$$
such that: 

i) $\,\omega_1\,$ is nonzero and null with respect to $\,\langle\,\,,\,\,\rangle,\,$ that is, 
$$\langle \omega_1 (X,Y), \omega_1 (Z,U) \rangle = 0,\;\forall\, X,Y,Z,U \in V=T_pM,$$ 
and

ii) $\,\omega_2\,$ is flat with $\,\mbox{dim}\,N\left(\omega_2\right) \geq n-\mbox{dim}\,W_2.$}

\vspace{0.2cm}

\noindent{Notice that} $\,\beta (Z,Z)\neq 0\,$ for all $Z\neq 0$ in $TM.$ Given $\,X\in RE(\beta ),$ there exists $\,Z\neq 0\,$ such that $\,Z\in \mbox{ker}\,\beta (X)$ due to $\,n>2d+3.$  Since $\,\beta\,$ is flat, $\,\beta (Z,V)\subset U(X)$ by Lemma 2.2 and $\,U(X)\,$ is isotropic, it holds that $\,\beta (Z,Z)\in S(\beta )\cap S(\beta )^{\perp}.$ Let $\,\upsilon_{i}=\gamma_{i}+b_{i}\xi +c_{i}\eta +\delta_{i},\,1 \leq i \leq \ell,$ with $\,\gamma_{i}\in T_{f}^{\perp}M\,$ and $\,\delta_{i}\in
T_{g}^{\perp}M,$ be a basis of the isotropic subspace $\,S(\beta )\cap S(\beta )^{\perp}.$ The vectors $\gamma_{i}+b_{i}\xi ,\,1\leq i\leq \ell ,$ are linearly independent. Otherwise, there are real numbers $\rho_i,$ with some of them different of zero, such that $\sum_{i}\rho_i\,(\gamma_{i}+b_{i}\xi )=0.$ Since the vectors $\upsilon_i,\,1\leq i\leq \ell ,$ are linearly independent, we have $\sum_i\,\rho_i\,(c_{i}\eta +\delta_{i})=\sum_{i}\rho_i\,\upsilon_i\neq 0.$ The metric in $\mbox{span}\,\{\eta\}\oplus T_g^\perp M$ being positive definite, we have obtained a contradiction because the vector $\sum_i\,\rho_i\,(c_{i}\eta +\delta_{i})$ is not isotropic. Analogously, we obtain that the vectors $c_{i}\eta +\delta_{i},\,1 \leq i \leq \ell ,$ are linearly independent. Without loss of generality, we can suppose that the vectors $\eta_i=c_i\eta+\delta_i,\,1\leq i\leq \ell ,$ are orthonormal. In this case, due to $\langle \upsilon_i,\upsilon_j\rangle =0,$ $1\leq i,j\leq\ell ,$ we have $\langle\xi_i,\xi_j\rangle =- \delta_{ij}$ for $\xi_i=\gamma_{i}+b_{i}\xi,\,1\leq i\leq\ell .$ Extend $\eta_1,\ldots ,\eta_\ell $ to an orthonormal basis $\eta_1,\ldots ,\eta_{d+2}$  of $\mbox{span}\{\eta\}\oplus T_g^\perp M,$ and $\xi_1,\ldots ,\xi_{\ell}$ to a basis $\xi_1,\ldots ,\xi_{d+1}$ of $\mbox{span}\{\xi\}\oplus T_f^\perp M$ that satisfies $\langle \xi_i,\xi_j\rangle =- \delta_{ij}.$ Define $\hat{\upsilon}_i =-\xi_i+\eta_i,\,1\leq i\leq \ell ,$
$$ 
W_{1} = \mbox{span} \{\upsilon_{1},\ldots ,\upsilon_{\ell},\, \hat{\upsilon} _{1},\ldots
,\hat{\upsilon} _{\ell} \},\;\;\;\;W_{2} = \mbox{span} \{ \xi_{\ell +1},\ldots ,\xi_{d+1},\,\eta_{\ell +1},\ldots ,\eta_{d+2}\} 
$$ 
and put
$$ 
\beta =\sum_{i=1}^{\ell} \phi_{i} \upsilon_{i} + \sum_{i=1}^{\ell} \psi_{i} \hat{\upsilon}_{i} +
\sum_{i=\ell +1}^{d+1} \phi_{i} \xi_{i} + \sum_{i=\ell +1}^{d+2} \psi_{i} \eta_{i}. 
$$ 
Note that $W=W_{1}\oplus W_{2}$ is an orthogonal decomposition of $W.$ For $1\leq i \leq \ell ,$ we have $\,\psi_{i} = \frac{1}{2}\langle \beta , \upsilon_{i} \rangle = 0$. Set 
$$ 
\omega_{1} = \sum_{i=1}^{\ell} \phi_{i} \upsilon_{i}\;\;\;\mbox{and}\;\;\;\omega_{2} =
\sum_{i=\ell +1}^{d+1} \phi_{i} \xi_{i} + \sum_{i=\ell +1}^{d+2} \psi_{i} \eta_{i}. 
$$ 
Since $\,\ell=\mbox{dim}\,S(\beta )\cap S(\beta )^{\perp} \geq 1,$ then $\,\omega_{1}\,$ is nonzero. It is easy to verify that $\,\omega_{1},\omega_{2}\,$ are symmetric
bilinear forms such that $\,\omega_{1}\,$ is null and $\,\omega_{2}\,$ is flat. In order to see that $\,S(\omega_{2})\,$ is nondegenerate, let $\,\sum_{i}\omega_{2}(X_{i},Y_{i})\in W_{2}\,$ be an arbitrary element in $\,S(\omega_{2})\cap S(\omega_{2})^{\perp}.$ For all $\,v,w \in V,$ we get
$$
\left\langle \sum_{i}\omega_{2}(X_{i},Y_{i}),\beta (v,w)\right\rangle =
\left\langle\sum_{i}\omega_{2}(X_{i},Y_{i}),\omega_{2}(v,w)\right\rangle =0. 
$$ 
Therefore, $\,\sum_{i}\omega_{2}(X_{i},Y_{i})\in S(\beta )\cap S(\beta )^{\perp} .$ Hence, $\sum_{i}\omega_{2}(X_{i},Y_{i})\in W_{1}.$ Thus, 
$$
\sum_{i}\omega_{2}(X_{i},Y_{i})\in W_{1}\cap W_{2}=\{0\}. 
$$ 
Since the subspace $\,S(\omega_{2})\,$ is nondegenerate and $\,d-\ell+1\leq 5,$ the inequality $\,\mbox{dim}\,N(\omega_{2}) \geq n-\mbox{dim}\,W_{2}\,$ is a consequence of the following result whose proof is part of the arguments for the Main Lemma 2.2 in (\cite{C-D}, pp. 968-974).
  
\begin{lemma} Let $\sigma\colon\;V_1\times V_1\to W^{(r,t)}$ be a nonzero flat symmetric bilinear form. Assume $\,r\leq 5\,$ and $\mbox{dim}\,N(\sigma)< \mbox{dim}\,V_{1}-(r+t).$ Then
$\,S(\sigma)\,$ is degenerate.
\end{lemma}

Now to conclude Lemma 2.3, that is $\ell =d+1,$ we proceed exactly as in proof of Assertion 3 in (\cite{Si}, p. 238).\end{proof}

\begin{lemma} There exist an orthonormal basis $\zeta_1,\ldots ,\zeta_d$ of $T_{f(p)}^\perp M$ and an orthonormal basis $\vartheta_1,\ldots, \vartheta_{d+1}$ of $T_{g(p)}^\perp M$ such that
\begin{eqnarray}
\omega_1&=&-\langle\alpha^G,\xi\rangle\,(\xi-\eta \sin\theta\,\cos\varphi+\vartheta_{d+1}\cos\theta\,\cos\varphi+\vartheta_1sin\varphi )\nonumber\\
&-&\langle \alpha^f,\zeta_1\rangle\,(\zeta_1+\eta\sin\theta\,\sin\varphi-\vartheta_{d+1}\cos\theta\,\sin\varphi+\vartheta_1\cos\varphi )\nonumber\\
&-&\sum_{j=2}^d \langle \alpha^f,\zeta_j\rangle\,(\zeta_j+\vartheta_j)\nonumber\\
&&\label{5}\\
\omega_2&=&\,\langle \alpha^G ,\eta\cos\theta +\vartheta_{d+1}\sin\theta\rangle\,(\eta\cos\theta +\vartheta_{d+1}\sin\theta ).\nonumber
\end{eqnarray}
\end{lemma}

\begin{proof} Here our notations are as in proof of Assertion. We have $W_2=\mbox{span}\{\eta_{d+2}\}$ due to $\ell=d+1$. Put $\eta_{d+2}=\eta\cos\theta +\vartheta_{d+1}\sin\theta .$ The vector $-\eta\sin\theta +\vartheta_{d+1}\cos\theta$ belongs to $\mbox{span}\{\eta_{i}\;:\;1\leq i\leq d+1\}$ and its orthogonal complement in $\mbox{span}\{\eta_i\;:\;1\leq i\leq d+1\}$ is the orthogonal complement of $\vartheta_{d+1}$ in $T_{g(p)}^\perp M.$ The vectors $\xi_i=\gamma_i+b_i\xi,\,1\leq i\leq d+1$, being linearly independent, span the space $\mbox{span}\{\xi\}\oplus T_{f(p)}^\perp M.$ Write $\xi =\sum_{i=1}^{d+1}\rho_i\xi_i$ and take the vector in $S(\beta )\cap S(\beta )^{\perp}$ given by 
$$
\sum_{i=1}^{d+1}\rho_i\upsilon_i =\xi+\cos\varphi\,(-\eta\sin\theta +\vartheta_{d+1}\cos\theta )+\vartheta_1\sin\varphi ,
$$
where $\vartheta_1\in T_g^\perp M$ is an unitary vector orthogonal to $\vartheta_{d+1}.$ The vector $-\sin\varphi\,(-\eta\sin\theta +\vartheta_{d+1}\cos\theta )+\vartheta_1\cos\varphi$ also belongs to $\mbox{span}\{\eta_{i}\;:\;1\leq i\leq d+1\}$ since it is orthogonal to $\eta_{d+2}.$ Take $\zeta_1$ as being the unitary vector in $T_f^\perp M$ such that 
$$
\zeta_1-\sin\varphi\,(-\eta\sin\theta +\vartheta_{d+1}\cos\theta )+\vartheta_1\cos\varphi\in S(\beta )\cap S(\beta )^{\perp}.
$$
Now extend $\zeta_1$ to an orthonormal basis $\zeta_1,\ldots ,\zeta_d$ of $T_f^\perp M$ and take the vectors $\vartheta_2,\ldots ,\vartheta_d$ in $T_g^\perp M$ such that $\zeta_i+\vartheta_i,\,2\leq i \leq d,$ belong to $S(\beta )\cap S(\beta )^{\perp}.$ Now it is not difficult to see that $\vartheta_1,\ldots ,\vartheta_{d+1}$ is an orthonormal basis of $T_g^\perp M$ and satisfies (\ref{5}).\end{proof}

Note also that in (\ref{5}) the form $\omega_1$ is a linear combination of vectors orthogonal to $\beta =\alpha^f\oplus \alpha^G.$ These orthogonality give us
\begin{eqnarray}
\langle\alpha^G,\xi\rangle &=& \langle\alpha^G ,\eta\rangle\sin\theta\,\cos\varphi -\langle \alpha^G ,\vartheta_{d+1}\rangle\cos\theta\,\cos\varphi -\langle\alpha^G,\vartheta_1\rangle\sin\varphi\nonumber\\
\langle\alpha^f,\zeta_1\rangle &=& -\langle\alpha^G,\eta\rangle\sin\theta\,\sin\varphi +\langle \alpha^G ,\vartheta_{d+1}\rangle\cos\theta\,\sin\varphi -\langle\alpha^G,\vartheta_1\rangle\cos\varphi\label{6}\\
-\langle\alpha^f,\zeta_j\rangle &=& \langle\alpha^G,\vartheta_j\rangle,\;2\leq j\leq d.\nonumber
\end{eqnarray}
The above first two equations imply that
\begin{equation}
\langle\alpha^G,\vartheta_1\rangle =-\langle\alpha^G,\xi\rangle\sin\varphi -\langle\alpha^f,\zeta_1\rangle\cos\varphi .\label{7}
\end{equation}

\begin{proof}[of Lemma 2.1] First observe that in (\ref{5}) we have $\sin\theta\,cos\varphi\neq -1.$ Otherwise, $G=\xi +\eta$ belongs to $S(\beta )\cap S(\beta )^{\perp}$ and, consequently,
$$
0=\langle G ,\beta\rangle =\langle G,\alpha^G\rangle=-\langle\,,\,\rangle
$$
which is a contradiction. Then, we can consider the orthonormal basis of $T_g^\perp M$ given by

\begin{eqnarray}
\overline{\mu}_{1}&=&\frac{1}{1+\sin\theta\,\cos\varphi}\,\left[\vartheta_{1}\,(\cos\varphi+\sin\theta )-\vartheta_{d+1}\,\sin\varphi\,\cos\theta\right],\;\;\;\mu_j=\vartheta_j,\,2\leq j\leq d,\nonumber\\
\overline{\mu}_{d+1}&=&\frac{1}{1+\sin\theta\,\cos\varphi}\,\left[\vartheta_1\,\sin\varphi\,\cos\theta+\vartheta_{d+1}\,(\cos\varphi+\sin\theta )\right].\label{8}
\end{eqnarray}
By (\ref{6}), for $2\leq j\leq d,$ it holds that
$$
\langle \alpha^G,\mu_j\rangle =\langle \alpha^G,\vartheta_j\rangle =-\langle \alpha^f,\zeta_j\rangle .
$$
Thus, we can write
\begin{eqnarray}
\alpha^G =-\langle \alpha^G,\xi\rangle\,(\xi+\eta)-\langle\,,\,\rangle\,\eta +\langle\alpha^G,\overline{\mu}_1\rangle\,\overline{\mu}_1-\sum_{j=2}^d\langle\alpha^f,\zeta_j\rangle\,\mu_j+\langle\alpha^G,\overline{\mu}_{d+1}\rangle\,\overline{\mu}_{d+1}.\label{9}
\end{eqnarray}
The equalities in (\ref{5}) give that
\begin{eqnarray*}
\langle\alpha^G,\overline{\mu}_1\rangle &=&\langle\beta,\overline{\mu}_1\rangle=\langle\omega_1+\omega_2,\overline{\mu}_1\rangle =\\
&-&\langle\alpha^G,\xi\rangle\,\sin\theta\,\sin\varphi-\frac{\langle \alpha^f,\zeta_1\rangle}{1+\sin\theta\,\cos\varphi}\left[1-\sin^2\theta\,\sin^2\varphi +\sin\theta\,\cos\varphi\right]\\
&-&\left[\langle \alpha^G ,\eta\cos\theta +\vartheta_{d+1}\sin\theta\rangle\cos\theta\right]\frac{\sin\theta\,\sin\varphi}{1+\sin\theta\,\cos\varphi}
\end{eqnarray*}
and, jointly with $-\langle\,,\,\rangle =\langle G,\alpha^G\rangle =\langle \xi +\eta ,\beta\rangle =\langle \xi +\eta ,\omega_1+\omega_2\rangle $, that
\begin{equation}
\langle \alpha^G ,\eta\cos\theta +\vartheta_{d+1}\sin\theta\rangle\,\cos\theta =-\langle\,,\,\rangle-\langle \alpha^G,\xi\rangle\,(1+\sin\theta\,\cos\varphi )+\langle \alpha^f,\zeta_1\rangle\,\sin\theta\,\sin\varphi.\label{10}
\end{equation}
Therefore,
\begin{equation}
\langle\alpha^G,\overline{\mu}_1\rangle =\frac{\sin\theta\,\sin\varphi }{1+\sin\theta\,\cos\varphi }\langle\,,\,\rangle -\langle \alpha^f,\zeta_1\rangle .\label{11}
\end{equation}
From (\ref{6}) and $A^G_{\xi +\eta}=-I,$ we deduce that
$$
\cos\varphi\langle\alpha^G,\xi\rangle -\sin\varphi\langle \alpha^f ,\zeta_{1}\rangle =-\sin\theta\langle\,,\,\rangle -\cos\theta\langle\alpha^G,\vartheta_{d+1}\rangle -\sin\theta\langle\alpha^G,\xi\rangle.
$$
So
$$
(\cos\varphi+\sin\theta)\langle\alpha^G,\xi\rangle =\sin\varphi\langle \alpha^f ,\zeta_{1}\rangle -\sin\theta\langle\,,\,\rangle -\cos\theta\langle\alpha^G,\vartheta_{d+1}\rangle.
$$
Multiplying the above equation by $\cos\varphi +\sin\theta$ and introducing $\overline{\mu}_{d+1}$ according to (\ref{8}), it is a straightforward calculation to see that 
\begin{eqnarray*}
(\cos\varphi +\sin\theta )^2\langle\alpha^G,\xi\rangle &=&\sin\varphi\,(\cos\varphi +\sin\theta )\langle\alpha^f,\zeta_1\rangle  
-\sin\theta\,(\cos\varphi +\sin\theta )\langle\,,\,\rangle \\
&-&\langle \alpha^G ,\overline{\mu}_{d+1}\rangle (1+\cos\varphi\,\sin\theta )\cos\theta +\langle \alpha^G ,\vartheta_{1}\rangle\cos^2\theta\,\sin\varphi .
\end{eqnarray*}
This and (\ref{7}) imply
\begin{eqnarray}
-\langle\alpha^G,\xi\rangle &=&-\frac{\sin\varphi\,\sin\theta}{1+\cos\varphi\,\sin\theta}\langle\alpha^f,\zeta_1\rangle +\frac{\sin\theta\,(\cos\varphi +\sin\theta )}{(1+\cos\varphi\,\sin\theta )^2}\langle\,,\,\rangle\nonumber\\
&+&\frac{\cos\theta}{1+\cos\varphi\,\sin\theta}\langle \alpha^G ,\overline{\mu}_{d+1}\rangle .\label{12}
\end{eqnarray}
The formulae (\ref{9}), (\ref{11}) e (\ref{12}) give us
\begin{eqnarray*}
\alpha^G &=&\langle\,,\,\rangle\left[\frac{\sin\theta\,(\cos\varphi +\sin\theta )}{(1+\cos\varphi\,\sin\theta )^2}(\xi +\eta)-\eta + \frac{\sin\varphi\,\sin\theta}{1+\cos\varphi\,\sin\theta}\overline{\mu}_1\right]\\
&+&\langle\alpha^f,\zeta_1\rangle\left[-\frac{\sin\varphi\,\sin\theta}{1+\cos\varphi\,\sin\theta}(\xi +\eta )-\overline{\mu}_1\right]-\sum_{j=2}^d \langle\alpha^f,\zeta_j\rangle\mu_j\\
&+&\langle \alpha^G ,\overline{\mu}_{d+1}\rangle\left[\overline{\mu}_{d+1}+\frac{\cos\theta}{1+\cos\varphi\,\sin\theta}(\xi +\eta )\right].
\end{eqnarray*}
From which it is not difficult to see that
\begin{eqnarray*}
\alpha^G &=&\langle\,,\,\rangle \frac{1}{1+\cos\varphi\,\sin\theta }[\xi -\eta\,\cos\varphi\,\sin\theta +\overline{\mu}_1\,\sin\varphi\,\sin\theta +\overline{\mu}_{d+1}\cos\theta ]\\
&-&\langle\alpha^f,\zeta_1\rangle\left[\frac{\sin\varphi\,\sin\theta}{1+\cos\varphi\,\sin\theta}(\xi +\eta )+\overline{\mu}_1\right]-\sum_{j=2}^d \langle\alpha^f,\zeta_j\rangle\mu_j\\
&+&\left\langle \alpha^G ,\overline{\mu}_{d+1}+\frac{\cos\theta}{1+\cos\varphi\,\sin\theta}(\xi +\eta )\right\rangle\left[\overline{\mu}_{d+1}+\frac{\cos\theta}{1+\cos\varphi\,\sin\theta}(\xi +\eta )\right].
\end{eqnarray*}
Now (\ref{3}) follows if we define
\begin{eqnarray*}
\mu_o &=&-\frac{1}{1+\cos\varphi\,\sin\theta }[\xi -\eta\,\cos\varphi\,\sin\theta +\overline{\mu}_1\,\sin\varphi\,\sin\theta +\overline{\mu}_{d+1}\cos\theta ],\\
\mu_1 &=&\frac{\sin\varphi\,\sin\theta}{1+\cos\varphi\,\sin\theta}(\xi +\eta )+\overline{\mu}_1,\;\;
\mu_{d+1}=\overline{\mu}_{d+1}+\frac{\cos\theta}{1+\cos\varphi\,\sin\theta}(\xi +\eta ).
\end{eqnarray*}
Notice that if $v\in V$ and $w\in N(\omega_2 ),$ then 
$$
\langle A^G_{\mu_{d+1}}v,w\rangle=\langle \alpha^G (v,w),\mu_{d+1}\rangle = \langle \beta (v,w),\mu_{d+1}\rangle =\langle \omega_1 (v,w),\mu_{d+1}\rangle =0 
$$
due to (\ref{5}), (\ref{8}) and (\ref{10}). Thus, $\mbox{rank}\,A_{\mu_{d+1}}^G\leq 1$. Recall that $\mbox{dim}\,N(\omega_2 )\geq n-1$ by Assertion in proof of Lemma 2.3 and $\ell=d+1$.\end{proof}

We point out that $p$ in Lemma 2.1 being arbitrary the decomposition (\ref{3}) holds on $M^n.$

\begin{lemma} At $p\in M^n$, it holds that $\mbox{dim}\,S\left(\alpha^G\right)\geq d+1$. Furthermore, $\mbox{dim}\,S\left(\alpha^G\right)= d+1$ at $p$ if and only if $A^G_{\mu_{d+1}}\equiv 0$ at $p$ and $\mbox{dim}\,S\left(\alpha^G\right)= d+2$ at $p$ if and only if $\mbox{rank}\,A^G_{\mu_{d+1}}=1$ at $p$.\end{lemma}

\begin{proof} Consider $S\left(\alpha^G\right)^\perp ,$ the orthogonal complement of $S\left(\alpha^G\right)$ in the nondegenerate $(d+3)$-dimensional vector space 
$$T_{G(p)}^\perp M=\mbox{span}\,\{G,\mu_o,\mu_1,\ldots ,\mu_{d+1}\}.$$ 
Note that the coordinates of an arbitrary vector $\mu\in T_{G(p)}^\perp M$ on the basis $G,\mu_o,\mu_1,\ldots ,\mu_{d+1}$ are given by
\begin{equation}
\mu=\langle \mu,\mu_o\rangle G+\langle \mu,G\rangle\mu_o +\sum_{i=1}^d\langle \mu,\mu_i\rangle \mu_i+\langle\mu,\mu_{d+1}\rangle \mu_{d+1}.\label{13}
\end{equation}
Now take $\mu\in S\left(\alpha^G\right)^\perp .$  Then, for all $v,\,w\in V,$ we have
\begin{equation}
0=\left\langle \alpha^G (w,v),\mu \right\rangle =\left\langle \left(\,A^f_{\sum_i \langle \mu_i ,\mu\rangle\,\zeta_i}-\langle \mu_o,\mu\rangle\,I+\langle \mu_{d+1},\mu\rangle\,A^G_{\mu_{d+1}}\,\right)(w),v\right\rangle \label{14}
\end{equation}
due to (\ref{3}). For $w\in \mbox{Ker}\,\left(A^G_{\mu_{d+1}}\right),$ we obtain
$$0=\left\langle \left(\,A^f_{\sum_i \langle \mu_i ,\mu\rangle\,\zeta_i}-\langle \mu_o,\mu\rangle\,I\,\right)(w),v\right\rangle ,\;\forall v\in T_pM.
$$
So
$$
\nu_1^c\geq \mbox{dim}\,\mbox{Ker}\left(A^f_{\sum_i \langle \mu_i ,\mu\rangle\,\zeta_i}-\langle \mu_o,\mu\rangle\,I\right)\geq  \mbox{dim}\,\mbox{Ker}\left(A^G_{\mu_{d+1}}\right)\geq n-1
$$
which contradicts the hypothesis on the $1$-conformal nullity of $f$ case $\sum_i \langle \mu_i ,\mu\rangle\,\zeta_i\neq 0.$ Then $\langle \mu_i ,\mu\rangle =0,\,1\leq i\leq d,$ and consequently $\langle \mu_o ,\mu\rangle = 0.$  Thus, $\mu$ belongs to $\mbox{span}\{\mu_o ,\mu_{d+1}\}$ due to (\ref{13}). From (\ref{14}) we deduce that $\mu_{d+1}\in  S\left(\alpha^G\right)^\perp$ if and only if $A^G_{\mu_{d+1}}\equiv 0.$ Since $\mu_o\in S\left(\alpha^G\right)^\perp$ and $\mbox{rank}\,A^G_{\mu_{d+1}}\leq 1$ we have proved Lemma 2.6.\end{proof}
As a consequence of Lemma 2.6, if we define 
\begin{equation}
\omega =-\langle\,\,,\,\,\rangle \mu_{o} +\sum_{j=1}^{d}\left\langle \alpha^{f} , \zeta_{j}\right\rangle\mu_{j}, \label{15}
\end{equation}
we have $\mbox{dim}\,S(\omega )=d+1,$ that is, 
$$S(\omega )=\mbox{span}\,\{\,\mu_o,\,\mu_1,\,\mu_2,\ldots ,\mu_d\,\}.$$
In relation to $S(\omega )$ we also claim that
\begin{equation}
S(\omega ) =\,\mbox{span}\{\,\omega (X,Y)\,:\,X,Y \in  \mbox{Ker}\,A^G_{\mu_{d+1}}\,\}.\label{16} 
\end{equation}
In fact, if we put $R=\,\mbox{span}\{\,\omega (X,Y)\,:\,X,Y \in  \mbox{Ker}\,A^G_{\mu_{d+1}}\,\},$ it suffices to verify that $\mbox{dim}\,R=d+1.$ For to see this we show that the orthogonal complement of $R$ in the nondegenerate $(d+2)$-dimensional vector space $\mbox{span}\,\{G,\,\mu_o,\,\mu_1,\,\mu_2,\ldots ,\mu_d\,\}$ has dimension one. Let $\mu=aG+b\mu_{o} +\sum_{j=1}^{d}a_j\mu_{j}$ be an arbitrary vector orthogonal to $R.$ Then, for all $X,Y \in  \mbox{Ker}\,A^G_{\mu_{d+1}},$ we can write
$$0=\langle \omega (X,Y),\mu\rangle=-a\langle X,Y\rangle +\sum_{j=1}^{d}\left\langle \alpha^{f}(X,Y) , \zeta_{j}\right\rangle a_{j}=
\left\langle \left(A^f_{\sum_{j=1}^{d}a_{j}\zeta_{j}}-aI\right)X,Y\right\rangle.$$
Consequently, if $\gamma=\sum_{j=1}^{d}a_{j}\zeta_{j}\neq 0$ then the nullity of $\alpha^f_{\mbox{span}\,\{\gamma\}}-\langle\,,\,\rangle a\frac{\gamma}{|\gamma|^2}$ is at least $n-2.$ This implies that $\nu^c_1\geq n-2$ and we have obtained a contradiction with our hypothesis. Thus, $a=0=a_j,\,1\leq j\leq d,$ and the claim have been proved.

\begin{proof}[of Theorem 1.2] Let $\mathcal{U}$ be the subset of $M^n$ constituted of the points $q$ so that $\mbox{dim}\,S\left(\alpha^G\right)$ is constant in a neighborhood of $q$. The subset $\mathcal{U}$ is open and dense in $M^n$. In fact, clearly $\mathcal{U}$ is open in $M^n$. For to see that $\mathcal{U}$ is dense in $M^n$, consider an arbitrary point $p$ in $M^n$. By Lemma 2.6, in each point $p$ of $M^n$, the dimension of $\,S\left(\alpha^G\right)$ is either $d+1$ or $d+2$. If $\mbox{dim}\,S\left(\alpha^G\right)=d+2$ at $p$, then $p$ belongs to $\mathcal{U}$ since $\mbox{dim}\,S\left(\alpha^G\right)$ does not decrease in a neighborhood of $p$ by continuity. If $\mbox{dim}\,S\left(\alpha^G\right)=d+1$ at $p$ and $p$ does not belong to $\mathcal{U}$, there is a sequence of points where $\mbox{dim}\,S\left(\alpha^G\right)$  is $d+2$, consequently a sequence of points in $\mathcal{U}$, that converges to $p$. Thus, $\mathcal{U}$ is dense in $M^n$. Now, for an arbitrary point $p\in \mathcal{U}$, we consider a connected neighborhood $U\subset \mathcal{U}$ of $p$ where $F$ is an embedding and $\mbox{dim}\,S\left(\alpha^G\right)$ is constant. We divide the proof of Theorem 1.2 in two cases.

\vspace{0.2cm}

\noindent{\bf Case I.} $\mbox{dim}\,S\left(\alpha^G\right)=d+2$ on $U$. In this case, $\mbox{rank}\,A^G_{\mu_{d+1}}\equiv 1$ on $U$ by Lemma 2.6 and
$$S\left(\alpha^G\right) =\mbox{span}\,\{\mu_o,\mu_1,\ldots ,\mu_{d+1}\} $$
due to (\ref{3})

\begin{lemma} In Case I we can choose $\mu_{d+1}$ such that the unitary vector field $\mu_{d+1}\in T_G^\perp M$ be smooth on $U$. \end{lemma}

\begin{proof}[of Lemma 2.7] In each point of $U$ we choose $\mu_{d+1}$ so that the unique nonzero eigenvalue of $A^G_{\mu_{d+1}}$ be positive. We affirm that with this choice $\mu_{d+1}$ is smooth. First we prove that for all continuous tangent vector fields $X,Y\in TU$ the function $\langle\alpha^G(X,Y),\mu_{d+1}\rangle$ is continuous. Since $\alpha^G$ is bilinear and symmetric it suffices to prove that $\langle\alpha^G(X,X),\mu_{d+1}\rangle$ is continuous for all continuous field  $X$ in $TU$. All eigenvalues of $A^G_{\mu_{d+1}}$ being nonnegative, it holds that 
\begin{equation}
\langle\alpha^G(X,X),\mu_{d+1}\rangle =\langle A^G_{\mu_{d+1}}X,X\rangle\geq 0.\label{17}
\end{equation} 
From (\ref{3}) it follows that
$$
\left\langle\alpha^G(Z,W),\mu_{d+1}\right\rangle^2 =\left|\alpha^G(Z,W)\right|^2-\left|\alpha^f(Z,W)\right|^2,\;\forall Z,W\in TM.
$$
Recall that $\mu_o$ has zero length. In particular, $\left\langle\alpha^G(X,X),\mu_{d+1}\right\rangle^2$ is continuous. So $\left\langle\alpha^G(X,X),\mu_{d+1}\right\rangle$ is continuous due to (\ref{17}). (Note that $\left\langle\alpha^G(Z,W),\mu_{d+1}\right\rangle^2$ is smooth when $Z$ and $W$ are smooth.) Now in a fixed $q_o\in U$ consider an orthonormal basis $e_1,\ldots ,e_n$ of $T_{q_o}M$ of eigenvectors of $A^G_{\mu_{d+1}}$ so that $A^G_{\mu_{d+1}}e_1=\rho e_1,\,\rho >0,$ and $A^G_{\mu_{d+1}}e_j=0,\,2\leq j\leq n.$ Extend $e_1,\ldots,e_n$ locally to a smooth orthonormal frame $E_1,\ldots,E_n$ of tangent vectors and define the local smooth fields $Y_1=E_1$ and $Y_j=E_1+E_j,\,2\leq j\leq n.$ Observe that in each point $q$ where the vectors $Y_1(q),\ldots ,Y_n(q)$ are defined they are linearly independent and so they are a basis of $T_qM$. Since $\mbox{dim}\,S\left(\alpha^G\right)=d+2$ on $U$ we can take vectors $\alpha^G(Y_{i_k},Y_{j_k}),\,1\leq k\leq d+2,$  that are a basis of $S\left(\alpha^G\right)$ in a neighborhood of $q_o$. Consider the locally defined continuous functions 
$$
\psi_k= \left\langle\alpha^G(Y_{i_k},Y_{j_k}),\mu_{d+1}\right\rangle,\;1\leq k\leq d+2.
$$  
As we have seen above the functions $\psi_k$ are continuous. We claim that each $\psi_k$ is smooth. In fact, at $q_o$ we have $\psi_k(q_o)=\rho >0.$ Then in a neighborhood of $q_o$ we can suppose that $\psi_k$ is positive since it is continuous. We have observed previously that $\psi_k^2$ is smooth. Therefore, $\psi_k$ is smooth since it is positive. Let us denote $a_1,\ldots ,a_{d+2}$ the coordinates functions of $\mu_{d+1}$ on the basis $\alpha^G(Y_{i_k},Y_{j_k}),\,1\leq k\leq d+2.$ Using (\ref{3}), we can write
\begin{eqnarray*}
\mu_{d+1}=\sum_{k=1}^{d+2}a_k\alpha^{G}(Y_{i_k},Y_{j_k})=&-&\left(\sum_{k=1}^{d+2}a_k\langle Y_{i_k},Y_{j_k}\rangle\right) \mu_{o} \\
&+&\sum_{j=1}^{d}\left\langle \sum_{k=1}^{d+2}a_k\alpha^{f}(Y_{i_k},Y_{j_k}) , \zeta_{j}\right\rangle\mu_{j}+ \left(\sum_{k=1}^{d+2}a_k\psi_k\right)\mu_{d+1}.
\end{eqnarray*}
So the functions $a_k$ satisfy the equations
$$
\sum_{k=1}^{d+2}a_k\langle Y_{i_k},Y_{j_k}\rangle =0,\;\; \sum_{k=1}^{d+2}a_k\alpha^{f}(Y_{i_k},Y_{j_k})=0\;\;\mbox{and}\;\;\sum_{k=1}^{d+2}a_k\psi_k=1.
$$
Taking a smooth orthonormal frame $\xi_1,\ldots ,\xi_d$ of $T_f^\perp M$ in a neighborhood of $q_o,$ we have that the functions $a_k$ satisfy the following system of $d+2$ linear equations  
$$
\sum_{k=1}^{d+2}a_k\langle Y_{i_k},Y_{j_k}\rangle =0,\;\;
\sum_{k=1}^{d+2}a_k\left\langle\alpha^{f}(Y_{i_k},Y_{j_k}),\xi_j\right\rangle=0,\,1\leq j\leq d,\;\;\mbox{and}\;\;\sum_{k=1}^{d+2}a_k\psi_k=1.
$$
The order $d+2$ matrix of the system is an invertible smooth matrix. That it is smooth follows from the smoothly of the functions $Y_l,\,\alpha^f$ and $\psi_k$ for all $l=1,\ldots ,n$ and $k=1,\ldots ,d+2.$ If we consider a vector $(c_1,\ldots ,c_{d+2})\in \Bbb R^{d+2}$ in the kernel of the system, then the equations bellow are satisfied
$$
\sum_{k=1}^{d+2}c_k\langle Y_{i_k},Y_{j_k}\rangle =0,\;\;
\sum_{k=1}^{d+2}c_k\left\langle\alpha^{f}(Y_{i_k},Y_{j_k}),\xi_j\right\rangle=0,\,1\leq j\leq d,\;\;\mbox{and}\;\;\sum_{k=1}^{d+2}c_k\psi_k=0.
$$
By (\ref{3}), we deduce that $\sum_{k=1}^{d+2}c_k\alpha^{G}(Y_{i_k},Y_{j_k})=0.$ Being the vectors $\alpha^{G}(Y_{i_k},Y_{j_k}),$ $1\leq k\leq d+2,$ linearly independent, we have $c_k=0$ for all $k.$ So the matrix of the system is invertible. Then, the functions $a_k$ are smooth and, consequently, $\mu_{d+1}$ is an unitary smooth vector field on $T_G^\perp U.$\end{proof}

\begin{lemma} The null vector field $\mu_{o}$ is smooth on $U$. \end{lemma}

\begin{proof} The proof is identical to one in Lemma 2.7, noticing that the coordinates $a_1,\ldots ,a_{d+2}$ of $\mu_{o}$ on the basis $\alpha^G(Y_{i_k},Y_{j_k}),\,1\leq k\leq d+2,$ are determined by
$$
\sum_{k=1}^{d+2}a_k\langle Y_{i_k},Y_{j_k}\rangle =-1,\;\;
\sum_{k=1}^{d+2}a_k\left\langle\alpha^{f}(Y_{i_k},Y_{j_k}),\xi_j\right\rangle=0,\,1\leq j\leq d,\;\;\mbox{and}\;\;\sum_{k=1}^{d+2}a_k\psi_k=0.
$$
\end{proof}

\begin{lemma} On $U$ the null vector field $\mu_{o}$ is parallel along $\mbox{Ker}\,A^G_{\mu_{d+1}}$.\end{lemma}

\begin{proof} We denote by $\tilde{\nabla}$ the connection of the Lorentz space $I\!\!L^{n+d+3}$ and by $\tilde{\nabla}^\perp$ the normal connection of the immersion $G.$ First note that the coordinate of $\tilde{\nabla}^\perp_X\mu_{o}$  in the direction $\mu_o$ is zero for all $X\in TU.$ In fact, we have $\tilde{\nabla}^\perp_XG=0$ since $\tilde{\nabla}_XG=G_*X$ is tangent. Now if we take derivatives on $\langle G,\mu_{o}\rangle=1$ in the direction $X,$ we obtain that $\langle G,\tilde{\nabla}^\perp_X\mu_{o}\rangle=0$ and the coordinate of $\tilde{\nabla}^\perp_X\mu_{o}$  in the direction $\mu_o$ is zero by (\ref{13}). Being $\mu_{o}$ a vector field of zero length, the vector $\tilde{\nabla}^\perp_X\mu_{o}$ has not component in the direction $G$ by (\ref{13}). Now we claim that $\tilde{\nabla}^\perp_X\mu_{o}$ also has not component on 
$L=\mbox{span}\,\{\,\mu_1,\,\mu_2,\ldots ,\mu_d\,\}.$ In fact, consider $q\in U$ and the linear map
$$\begin{tabular}{ccc}
$T_qM$&$\stackrel{\Psi}{\longrightarrow}$&$\mbox{span}\,\left\{\,\mu_1,\ldots \mu_d,\,\mu_{d+1}\,\,\right\}$\\
$X $&$   \longrightarrow                 $&$\tilde{\nabla}^\perp_X\mu_{o}$.
\end{tabular}$$
Using the Codazzi's equations for $\alpha^G,$ the compatibility of the normal connection of the immersion $G$ with the metric in $T_G^\perp M$ and that $\left\langle \alpha^G,\mu_o\right\rangle\equiv 0,$ it is a straightforward calculation to see that
\begin{equation}
\left\langle \alpha^G(X,Y),\tilde{\nabla}^\perp_Z\mu_{o}\right\rangle=\left\langle \alpha^G(Z,Y),\tilde{\nabla}^\perp_X\mu_{o}\right\rangle,
\;\forall X,Y,Z\in T_qM.\label{18}
\end{equation}
This equation is equivalent to 
\begin{equation}
A^G_{\tilde{\nabla}^\perp_Z\mu_{o}}X=A^G_{\tilde{\nabla}^\perp_X\mu_{o}}Z ,\;\forall X,Z\in T_qM.\label{19}
\end{equation}
Let $\Psi_L$ be the component of $\Psi$ on  $L=\mbox{span}\,\{\,\mu_1,\,\mu_2,\ldots ,\mu_d\,\},$ that is, 
$$
\Psi_L(Z)=\sum_{i=1}^d\left\langle \tilde{\nabla}^\perp_Z\mu_{o},\mu_i\right\rangle\mu_i,\;\forall Z\in T_qM.
$$
Suppose that $\mbox{dim}\,\left(\mbox{Im}\,\Psi_L\right)=r$ and $r\geq 1.$ Consider a basis $\Psi_L\left(Z_1\right),\ldots ,\Psi_L\left(Z_r\right)$ of $\left(\mbox{Im}\,\Psi_L\right).$  Observe that $d\geq r\geq \mbox{dim}\,\left(\mbox{Im}\,\Psi\right)-1.$ Taking  $\xi_j=\sum_{i=1}^d\langle \Psi_L\left(Z_j\right), \mu_i\rangle \zeta_i,\,1\leq j\leq r,$ and using (\ref{3}) and (\ref{18}), we obtain that
\begin{equation}
A^f_{\xi_j}X=A^G_{\tilde{\nabla}^\perp_X\mu_{o}}Z_j -\rho\langle X,E\rangle\left\langle \tilde{\nabla}^\perp_{Z_j}\mu_{o}, \mu_{d+1} \right\rangle E,\label{20}
\end{equation}
being $E$ an unitary eigenvector such that $A^G_{\mu_{d+1}}E=\rho\,E$, $X\in T_qM$ and $j,\,1\leq j\leq r$. For each $j,\,1\leq j\leq r,$ consider the linear map $\psi_j\;\colon T_qM\to \mbox{span}\,\{E\}$ given by 
$$
\psi_j(X)=-\rho\langle X,E\rangle\left\langle \tilde{\nabla}^\perp_{Z_j}\mu_{o}, \mu_{d+1} \right\rangle E=A^f_{\xi_j}X-A^G_{\tilde{\nabla}^\perp_X\mu_{o}}Z_j.
$$
Notice that if $X\in \mbox{Ker}\,\Psi\cap\left(\bigcap_{j=1}^r\mbox{Ker}\,\psi_j\right)$ then $X\in \bigcap_{j=1}^r\mbox{Ker}\,A^f_{\xi_j}$ due to (\ref{20}). We have $\mbox{dim}\,\left(\bigcap_{j=1}^r\mbox{Ker}\,\psi_j\right)\geq n-r.$ Then, for the $r$-dimensional space $\mathcal{L}=\mbox{span}\,\{\,\xi_1,\ldots,\xi_r\,\},$ using the formula
$$ 
\mbox{dim}\,(\,L_{1} + L_{2}\,) = \mbox{dim}\,L_{1} + \mbox{dim}\,L_{2} - \mbox{dim}\,L_{1} \cap L_{2}, 
$$ 
valid for any two finite dimensional vector subspaces of any vector space, we obtain
\begin{eqnarray*}
\nu^c_r&\geq& \mbox{dim}\,N\left(\alpha_{\mathcal{L}}^f\right)= \mbox{dim}\,\left(\bigcap_{j=1}^r\mbox{Ker}\,A^f_{\xi_j}\right)\\
&\geq& \mbox{dim}\,\left[\mbox{Ker}\,\Psi\cap\left(\bigcap_{j=1}^r\mbox{Ker}\,\psi_j\right)\right]\geq \left(n-\mbox{dim}\,\left(\mbox{Im}\,\Psi\right)\right)+(n-r)-n\geq n-2r-1
\end{eqnarray*}
which is in contradiction with our hypothesis on the $r$-conformal nullity. Then, $r=0$ and $\tilde{\nabla}^\perp_Z\mu_{o}=\left\langle \tilde{\nabla}^\perp_Z\mu_{o},\mu_{d+1}\right\rangle \mu_{d+1}$ for all $Z\in T_qM.$ Now from (\ref{19}) we deduce that
$$
\left\langle \tilde{\nabla}^\perp_Z\mu_{o},\mu_{d+1}\right\rangle A^G_{\mu_{d+1}}X=\left\langle \tilde{\nabla}^\perp_X\mu_{o},\mu_{d+1}\right\rangle A^G_{\mu_{d+1}}Z
$$
for all $X,\,Z\in T_qM.$ Thus, if we choose $X=E$ and $Z\in \mbox{Ker}\,A^G_{\mu_{d+1}},$ we have
$$\rho\left\langle \tilde{\nabla}^\perp_Z\mu_{o},\mu_{d+1}\right\rangle E=0$$
for all $Z\in \mbox{Ker}\,A^G_{\mu_{d+1}}.$ Hence $\tilde{\nabla}^\perp_Z\mu_{o}=0$ for all $Z\in \mbox{Ker}\,A^G_{\mu_{d+1}}.$ Recall that our choice for $\mu_{d+1}$ is so that the unique nonzero eigenvalue $\rho$ of the smooth linear map $A^G_{\mu_{d+1}}$ is positive.\end{proof}

\begin{lemma} On $U$ the unitary vector field $\mu_{d+1}$ is parallel along $\mbox{Ker}\,A^G_{\mu_{d+1}}.$\end{lemma}

\begin{proof} Notice that the coordinate of $\tilde{\nabla}^\perp_X\mu_{d+1}$  in the direction $\mu_o$ is zero due to 
$$
\langle G,\tilde{\nabla}^\perp_X\mu_{d+1}\rangle=-\langle \tilde{\nabla}^\perp_XG,\mu_{d+1}\rangle=-\langle G_*X,\mu_{d+1}\rangle=0.
$$
By Lemma 2.9, for $X\in \mbox{Ker}\,A^G_{\mu_{d+1}},$ $\tilde{\nabla}^\perp_X\mu_{d+1}$ also has not component in the direction $G$ since $\langle \mu_o,\tilde{\nabla}^\perp_X\mu_{d+1}\rangle=-\langle \tilde{\nabla}^\perp_X\mu_o,\mu_{d+1}\rangle=0.$ Being $\mu_{d+1}$ an unitary vector field, the vector $\tilde{\nabla}^\perp_X\mu_{d+1}$ has not component in the direction $\mu_{d+1}.$ Thus, we only need proving that $\tilde{\nabla}^\perp_X\mu_{d+1}$ also has not component on 
$L=\mbox{span}\,\{\,\mu_1,\,\mu_2,\ldots ,\mu_d\,\}$ for all $X\in \mbox{Ker}\,A^G_{\mu_{d+1}}.$ Consider the linear map
$$\begin{tabular}{ccc}
$\mbox{Ker}\,A^G_{\mu_{d+1}}$&$\stackrel{\Phi}{\longrightarrow}$&$\mbox{span}\,\left\{\,\mu_1,\ldots \mu_d\,\right\}$\\
$                    X      $&$   \longrightarrow              $&$ \tilde{\nabla}^\perp_X\mu_{d+1}$.
\end{tabular}$$ 
Our choice for $\mu_{d+1}$ so that the unique nonzero eigenvalue $\rho$ of the smooth linear map $A^G_{\mu_{d+1}}$ is positive, for a well known argument, implies that $\rho$ is smooth. Simple calculations shows that in an arbitrary point of $U$ the $(n-1)$-dimensional distribution $\mbox{Ker}\,A^G_{\mu_{d+1}}$ is given by
$$
\mbox{Ker}\,A^G_{\mu_{d+1}}=\left\{\,\left(A^G_{\mu_{d+1}}-\rho\,I\right)W\;|\;W\in TU\,\right\}.
$$
So $\mbox{Ker}\,A^G_{\mu_{d+1}}$ is a differentiable distribution on $U.$ Therefore, if in some point $q_o$ of $U$ we consider  $X\in \mbox{Ker}\,A^G_{\mu_{d+1}},$ we can take in a neighborhood of $q_o$ in $U$ a differentiable extension of $X$ that lies on $\mbox{Ker}\,A^G_{\mu_{d+1}}$. With these observations, using the Codazzi's equations for $\alpha^G$ and the compatibility of the normal connection of the immersion $G$ with the metric in $T_G^\perp M,$ it is a straightforward calculation to see that
$$
A^G_{\Phi (Z)}X=A^G_{\Phi (X)}Z +\rho\langle [X,Z], E\rangle E,\;\mbox{for all}\; X,\,Z\in \mbox{Ker}\,A^G_{\mu_{d+1}}.
$$
Suppose that $\mbox{dim}\,\left(\mbox{Im}\,\Phi\right)=r$ and $r\geq 1.$ At this point we proceed like in proof of Lemma 2.9, with  $\mbox{Ker}\,A^G_{\mu_{d+1}}$ instead of $T_qM,$ for deduce that $\mbox{dim}\,\left(\mbox{Im}\,\Phi\right)=0$ and Lemma 2.10 has been proved.\end{proof}

Now we consider the linear isometry $\tau\colon\;T_F^\perp U\to \left(\mbox{span}\,\{\mu_{d+1}\}\right)^\perp\subset T_G^\perp U$ given by 
$$
\tau(F)=G,\;\tau (\zeta)=\mu_o\;\;\;\mbox{and}\;\;\;\tau(\zeta_i)=\mu_i,\,1\leq i\leq d.
$$
Notice that $\tau$ is smooth since 
\begin{equation}
T_F^\perp M=\mbox{span}\,\{F\}\oplus S\left(\alpha^F\right),\;\;
\left(\mbox{span}\,\{\mu_{d+1}\}\right)^\perp=\mbox{span}\,\{\,G\,\}\oplus S(\omega)\label{21}
\end{equation}
and $\tau \left(\alpha^F(X,Y)\right)=\omega(X,Y)$, $\forall X,\,Y\in TM$, due to (\ref{1}) and  (\ref{15}). Observe also that our hypothesis on the $1$-conformal nullity of $f$ implies that $S\left(\alpha^F\right)$ is a $(d+1)$-dimensional subspace of $T_F^\perp M$.

Since the $(n-1)$-dimensional distribution $\mbox{Ker}\,A^G_{\mu_{d+1}}$ is smooth on $U$, if $E$ is an unitary  differentiable vector field orthogonal to $\mbox{Ker}\,A^G_{\mu_{d+1}}$ then $E$ is in each point an eigenvector of $A^G_{\mu_{d+1}}$ corresponding to the eigenvalue $\rho$. Consequently, $E$ can be chosen smooth on $U$. Consider the vector bundle isometry $T\colon\;T_F^\perp U\oplus \mbox{span}\,\{E\}\to \left(\mbox{span}\,\{\mu_{d+1}\}\right)^\perp \oplus \mbox{span}\,\{E\}$ defined as $T(\delta +cE)=\tau (\delta)+cE$ for all $\delta +cE\in  T_F^\perp U\oplus \mbox{span}\,\{E\}.$ Take the subbundle $\Lambda\subset \left(\mbox{span}\,\{\mu_{d+1}\}\right)^\perp \oplus \mbox{span}\,\{E\}$ whose bundles are the orthogonal complement of the subspace generated by $\tilde{\nabla}_E\mu_{d+1}=-\rho E+\tilde{\nabla}_E^\perp\mu_{d+1}.$ We observe that, being $\left(\mbox{span}\,\{\mu_{d+1}\}\right)^\perp \oplus \mbox{span}\,\{E\}$ a $(d+3)$-dimensional nondegenerate vector bundle, $\Lambda$ is a $(d+2)$-dimensional nondegenerate subbundle and that $\Lambda$ is transversal to $TU.$ Now define the $(d+2)$-dimensional nondegenerate subbundle $\Omega$ of $T_F^\perp U\oplus \mbox{span}\,\{E\},$ transversal to $TU,$ by setting 
$T(\Omega)=\Lambda.$ Let $\mathcal{F}\colon\;\Omega \to I\!\!L^{n+d+2}$ be defined by $\mathcal{F}(q,\vartheta)=F(q)+\vartheta$ for all $q\in U$ and all $\vartheta$ in the fibre $\Omega_q.$  Since the fibres of $\Omega$ are transversal to $TU$ and $F$ is an embedding on $U$, $\mathcal{F}$ is a diffeomorphism from a neighborhood $\Omega_0$ of the zero section in $\Omega$ into a neighborhood $\mathcal{W}$ of $F(U)$ in $I\!\!L^{n+d+2}.$ Consider $\Omega_0$ endowed with the flat metric induced by $\mathcal{F}.$ Restrict to $\Omega_0,$
$\mathcal{F}$ become an isometry onto $\mathcal{W}.$  From now on $\mathcal{F}$ will stand for this restriction. Now let $\mathcal{G}\colon\;\Omega_0 \to I\!\!L^{n+d+3}$ be defined by $\mathcal{G}(q,\vartheta)=G(q)+T(\vartheta)$ for all $q\in U$ and $\vartheta \in \Omega_q.$ We claim that $\mathcal{G}$ is an isometric immersion. In fact, for all  differentiable curve $\theta(t)=(\gamma (t),\delta (\gamma(t))+c(t)E(\gamma (t)))\in \Omega_0$ we have $\mathcal{G}(\theta (t))=G(\gamma(t))+\tau(\delta (\gamma(t)))+c(t)E(\gamma (t)).$ Consequently,
\begin{eqnarray}
\left(\mathcal{G}_{*}\right)_{\theta(0)}\theta^\prime(0)&=& \left(G_{*}\right)_{\gamma(0)}\gamma^\prime(0)+\tilde{\nabla}_{\gamma^\prime(0)} \tau(\delta)+\tilde{\nabla}_{\gamma^\prime(0)} cE\label{22}\\
&=&\left(G_{*}\right)_{\gamma(0)}\left[\gamma^\prime(0)-A^G_{\tau(\delta)}\gamma^\prime(0)+\nabla_{\gamma^\prime(0)} cE\right]
+\tilde{\nabla}^\perp_{\gamma^\prime(0)}\tau(\delta) +\alpha^G(\gamma^\prime(0),cE).\nonumber
\end{eqnarray}

\noindent{Now} if we put $\gamma^\prime(0)=aE+v,$ being $v$ the component of $\gamma^\prime(0)$ in $\mbox{Ker}\,A^G_{\mu_{d+1}},$ we have
\begin{eqnarray}
&\,&\left\langle \tilde{\nabla}^\perp_{\gamma^\prime(0)}\tau(\delta) +\alpha^G(\gamma^\prime(0),cE),\mu_{d+1}\right\rangle=
\left\langle \tilde{\nabla}_{\gamma^\prime(0)}\left(\tau(\delta) + cE\right),\mu_{d+1}\right\rangle\label{23}\\
&=&-\left\langle \tau(\delta) + cE,\tilde{\nabla}_{\gamma^\prime(0)}\mu_{d+1}\right\rangle\nonumber
=-\left\langle \tau(\delta) + cE,a\tilde{\nabla}_E\mu_{d+1}-A^G_{\mu_{d+1}}v+\tilde{\nabla}^\perp_v\mu_{d+1}\right\rangle =0.\nonumber
\end{eqnarray}

\noindent{Recall} that $\tau(\delta) + cE\in \Lambda$ is orthogonal to $\tilde{\nabla}_E\mu_{d+1}$ and that, by Lemma 2.10, $\tilde{\nabla}^\perp_v\mu_{d+1}=0.$ For to finish the proof of the claim we will need of the following 

\begin{lemma} For $X\in TU$ and $\mu \in L=\left(\mbox{span}\,\{\mu_{d+1}\}\right)^\perp,$ let $\nabla^\prime_X\mu$ be given by $\nabla^\prime_X\mu=\tau\left({}^F\nabla_X^\perp\tau^{-1}(\mu)\right),$ being ${}^F\nabla^\perp$ the normal connection of the immersion $F,$  and let $\hat{\nabla}_X\mu$ be the component of $\tilde{\nabla}^\perp_X\mu$ on $L.$ Then, it holds that $\nabla^\prime_X\mu=\hat{\nabla}_X\mu.$ \end{lemma}

\begin{proof} For a fixed $X\in TU,$ we define the linear map $K(X)\;\colon L\to L$ by $K(X)\mu =\nabla^\prime_X\mu-\hat{\nabla}_X\mu,\;\forall \mu \in L.$ Since $L$ is nondegenerate, for to prove that $K(X)\equiv 0$ it suffices to prove, for all $W,\,V,\,Y,\,Z\in \mbox{Ker}\,A^G_{\mu_{d+1}},$ the following relations
$$
\langle K(X)G,G\rangle=\langle K(X)G,\omega(Y,Z)\rangle=\langle G,K(X)\omega(Y,Z)\rangle=\langle K(X)\omega(Y,Z),\omega(W,V)\rangle=0,
$$

\noindent{due to} (\ref{16}) and (\ref{21}). First we note that $K(X)G=0$ because ${}^F\nabla_X^\perp F =\tilde{\nabla}^\perp_X G =0.$ Also, for all $Y,\,Z\in \mbox{Ker}\,A^G_{\mu_{d+1}},$ we have
\begin{eqnarray*}
&\,&\langle G,K(X)\omega(Y,Z)\rangle =\langle G,\nabla^\prime_X\omega(Y,Z)\rangle - \langle G,\hat{\nabla}_X\omega(Y,Z)\rangle\\ 
&=&\left\langle \tau(F),\tau\left({}^F\nabla_X^\perp\alpha^F(Y,Z)\right)\right\rangle 
-\langle G,\tilde{\nabla}^\perp_X\omega(Y,Z)\rangle
=\left\langle F,{}^F\nabla_X^\perp\alpha^F(Y,Z)\right\rangle -\left\langle G,\tilde{\nabla}^\perp_X\alpha^G(Y,Z)\right\rangle\\
&=&X\left\langle F,\alpha^F(Y,Z)\right\rangle -X\left\langle G,\alpha^G(Y,Z)\right\rangle
=-X\langle Y,Z\rangle +X\langle Y,Z\rangle=0.
\end{eqnarray*}

\noindent{Now} we verify that holds the relation 
$$
K(X)\omega(Y,Z)=K(Y)\omega(X,Z),\;\mbox{for all}\; X\in TU\;\mbox{and all}\;Y,\,Z\in \mbox{Ker}\,A^G_{\mu_{d+1}}.
$$
In fact, for all $\mu\in L,$ we have
\begin{eqnarray*}
&\,&\langle K(X)\omega(Y,Z),\mu \rangle=\left\langle\tau\left({}^F\nabla_X^\perp\alpha^F(Y,Z)\right)-\tilde{\nabla}^\perp_X\alpha^G(Y,Z),\mu\right\rangle\\
&=& \left\langle\tau\left(\left({}^F\nabla_X^\perp\alpha^F\right)(Y,Z)\right)+\tau\left(\alpha^F\left(\nabla_X Y,Z\right)\right)+\tau\left(\alpha^F\left( Y,\nabla_X Z\right)\right)\right.\\
&-&\left.\left(\tilde{\nabla}_X^\perp\alpha^G\right)(Y,Z)-\alpha^G\left(\nabla_X Y,Z\right)
-\alpha^G\left( Y,\nabla_X Z\right),\mu\right\rangle\\
&=&\left\langle\tau\left(\left({}^F\nabla_X^\perp\alpha^F\right)(Y,Z)\right)-\left(\tilde{\nabla}_X^\perp\alpha^G\right)(Y,Z),\mu\right\rangle\\ 
&=&\left\langle\tau\left(\left({}^F\nabla_Y^\perp\alpha^F\right)(X,Z)\right)-\left(\tilde{\nabla}_Y^\perp\alpha^G\right)(X,Z),\mu\right\rangle\\
&=&\left\langle\tau\left({}^F\nabla_Y^\perp\alpha^F(X,Z)\right)-\tau\left(\alpha^F\left(\nabla_Y X,Z\right)\right)
-\tau\left(\alpha^F\left( X,\nabla_Y Z\right)\right)-\tilde{\nabla}_Y^\perp\alpha^G(X,Z)\right.\\
&+&\left.\alpha^G\left(\nabla_Y X,Z\right)
+\alpha^G\left( X,\nabla_Y Z\right),\mu\right\rangle=\left\langle\tau\left({}^F\nabla_Y^\perp\alpha^F(X,Z)\right)-\tilde{\nabla}^\perp_Y\alpha^G(X,Z),\mu\right\rangle\\
&=&\langle K(Y)\omega(X,Z),\mu \rangle.
\end{eqnarray*}

\noindent{Above} we have used that $\omega=\tau\left(\alpha^F\right),$ $\alpha^G(X,Z)=\omega(X,Z),$ for any $X\in TM,$ $Z\in \mbox{Ker}\,A^G_{\mu_{d+1}}$ and that $\alpha^F$ and $\alpha^G$ satisfy the Codazzi's equations. It is a straightforward calculation to see that $K(X)$ satisfies
$$
\langle K(X)\omega(Y,Z),\omega(W,V) \rangle =-\langle \omega(Y,Z),K(X)\omega(W,V) \rangle\;\mbox{for all}\;W,\,V,\,X,\,Y,\,Z\in TM.
$$
Now if $W,\,V,\,Y,\,Z\in \mbox{Ker}\,A^G_{\mu_{d+1}}$ and $X$ is an arbitrary tangent vector, then
\begin{eqnarray*}
\langle K(X)\omega(Y,Z),\omega(W,V) \rangle &=&\langle K(Y)\omega(X,Z),\omega(W,V)\\
&=&-\langle \omega(X,Z),K(Y)\omega(W,V)\rangle=-\langle \omega(X,Z),K(W)\omega(Y,V)\rangle\\
&=&\langle K(W)\omega(X,Z),\omega(Y,V)\rangle=\langle K(X)\omega(W,Z),\omega(Y,V)\rangle\\
&=&\langle K(Z)\omega(X,W),\omega(Y,V)\rangle=-\langle \omega(X,W),K(Z)\omega(Y,V)\rangle\\
&=&-\langle \omega(X,W),K(V)\omega(Y,Z)\rangle=\langle K(V)\omega(X,W),\omega(Y,Z)\rangle\\
&=&\langle K(X)\omega(W,V),\omega(Y,Z)\rangle=-\langle \omega(W,V),K(X)\omega(Y,Z)\rangle.
\end{eqnarray*}

\noindent{Therefore}, Lemma 2.11 has been proved.\end{proof}

Now, due to (\ref{23}), the equality (\ref{22}) becomes
\begin{eqnarray*}
\left(\mathcal{G}_{*}\right)_{\theta(0)}\theta^\prime(0)&=& 
\left(G_{*}\right)_{\gamma(0)}\left[\gamma^\prime(0)-A^G_{\tau(\delta)}\gamma^\prime(0)+\nabla_{\gamma^\prime(0)} cE\right]
+\tilde{\nabla}^\perp_{\gamma^\prime(0)}\tau(\delta) +\alpha^G(\gamma^\prime(0),cE)\\
&=&\left(G_{*}\right)_{\gamma(0)}\left[\gamma^\prime(0)-A^G_{\tau(\delta)}\gamma^\prime(0)+\nabla_{\gamma^\prime(0)} cE\right]
+\hat{\nabla}_{\gamma^\prime(0)}\tau(\delta) +\omega(\gamma^\prime(0),cE)\\
&=&\left[\left(G_{*}\right)_{\gamma(0)}\circ \left(\left(F_{*}\right)_{\gamma(0)}\right)^{-1}\right]\circ\left(F_{*}\right)_{\gamma(0)}\left[\gamma^\prime(0)-A^F_\delta\gamma^\prime(0)+\nabla_{\gamma^\prime(0)} cE\right]\\
&+&\tau\left[{}^F\nabla_{\gamma^\prime(0)}^\perp\delta +\alpha^F\left(\gamma^\prime(0),cE\right)\right].
\end{eqnarray*}

\noindent{Observe} that $A^G_{\tau(\delta)}=A^F_\delta .$ Then, $\left|\left(\mathcal{G}_{*}\right)_{\theta(0)}\theta^\prime(0)\right|^2=\left|\left(\mathcal{F}_{*}\right)_{\theta(0)}\theta^\prime(0)\right|^2$ for all differentiable curve $\theta (t)$ in $\Omega.$ This finish the proof of the claim.

Finally, taking $\mathcal{T}=\mathcal{G}\circ \mathcal{F}^{-1},$ we obtain an isometric immersion from a neighborhood of $F(U)$ in $I\!\!L^{n+d+2}$ into a neighborhood of $G(U)$ in $I\!\!L^{n+d+3}$ such that 
$$
\mathcal{T}(F(q))=\mathcal{T}(\mathcal{F}(q,0))=\mathcal{G}(q,0)=G(q),\,\forall q\in U.
$$
According to previous observations $\mathcal{T}$ induces a conformal immersion $\Gamma$ from a neighborhood of $f(U)$ in $\Bbb R^{n+d}$ into a neighborhood of $g(U)$ in $\Bbb R^{n+d+1}$ such that $g=\Gamma\circ f.$ This prove Theorem 1.2 in Case I.

\vspace{0.2cm}

\noindent{\bf Case II.} $\mbox{dim}\,S\left(\alpha^G\right)=d+1$ on $U$. In this case, by Lemma 2.6, $\mbox{rank}\,A^G_{\mu_{d+1}}\equiv 0$ on $U$ and,consequently,
$$S\left(\alpha^G\right) =\mbox{span}\,\{\mu_o,\mu_1,\ldots ,\mu_d\}.$$
by (\ref{3}).
 
The linear isometry $\tau\colon\;T_F^\perp U\to \left(\mbox{span}\,\{\mu_{d+1}\}\right)^\perp$ given by 
$$
\tau(F)=G,\;\tau (\zeta)=\mu_o\;\;\;\mbox{and}\;\;\;\tau(\zeta_i)=\mu_i,\,1\leq i\leq d,
$$
now satisfies $\tau (\alpha^F(X,Y))=\alpha^G(X,Y),\,\forall X,\,Y\in TU.$ Thus, $\tau$ is smooth. Then $\mu_o$ also is smooth. Since $\mbox{span}\,\{\mu_{d+1}\}$ is the orthogonal complement of $\mbox{span}\,\{G\}\oplus S\left(\alpha^G\right)$ in $T_G^\perp U,$ the unitary vector field $\mu_{d+1}$ can be choose smooth. Define 
$$\begin{tabular}{ccc}
$TU$&$\stackrel{\Psi}{\longrightarrow}$&$\mbox{span}\,\left\{\,G,\,\mu_1,\ldots \mu_d\,\right\}$ \\
$                  X      $&$   \longrightarrow                  $&$ \tilde{\nabla}^\perp_X\mu_{d+1}$,
\end{tabular}$$ 
Using the Codazzi's equations for $\alpha^G,$ the compatibility of the normal connection of the immersion $G$ with the metric in $T_G^\perp M$ and that $\left\langle \alpha^G,\mu_{d+1}\right\rangle\equiv 0,$ it is not difficult to see that
\begin{equation}
\left\langle \alpha^G(X,Y),\tilde{\nabla}^\perp_Z\mu_{d+1}\right\rangle=\left\langle \alpha^G(Z,Y),\tilde{\nabla}^\perp_X\mu_{d+1}\right\rangle,
\;\forall q\in U,\;\forall X,Y,Z\in T_qM,\label{24}
\end{equation}
that is,
\begin{equation}
A^G_{\tilde{\nabla}^\perp_Z\mu_{d+1}}X=A^G_{\tilde{\nabla}^\perp_X\mu_{d+1}}Z ,\;\forall q\in U,\;\forall X,Z\in T_qM.\label{25}
\end{equation}
At $q\in U$, let us denote $S$ the subspace of $T_{f(q)}^\perp M$ given by
$$
S=\left\{\,\sum_{i=1}^d\left\langle \tilde{\nabla}^\perp_Z\mu_{d+1},\mu_i\right\rangle \zeta_i\;:\;Z\in T_qM\,\right\}.
$$
Suppose that $\mbox{dim}\,S=r$ and $r\geq 1.$ Take a basis of $S$ given by
$\xi_j=\sum_{i=1}^d\left\langle \tilde{\nabla}^\perp_{Z_j}\mu_{d+1},\mu_i\right\rangle \zeta_i,$ $1\leq j\leq r.$ It holds that $d\geq r\geq \mbox{dim}\,\left(\mbox{Im}\,\Psi\right)-1.$ From (\ref{3}) and (\ref{24}) we deduce that
\begin{equation}
A^f_{\xi_j}X-\left\langle \mu_{o}, \tilde{\nabla}^\perp_{Z_j}\mu_{d+1} \right\rangle X=A^G_{\tilde{\nabla}^\perp_X\mu_{d+1}}Z_j ,\label{26}
\end{equation}
for all $q\in U$, $X\in T_qM$ and $j,\,1\leq j\leq r.$ Define $\gamma\in S$ by $\langle \gamma, \xi_j\rangle=\left\langle \mu_{o}, \tilde{\nabla}^\perp_{Z_j}\mu_{d+1} \right\rangle,$ $1\leq j\leq r.$ If $X\in \left(\mbox{Ker}\,\Psi\right)$ then $X\in N\left(\alpha_{S}-\langle\,,\,\rangle\gamma\right)$ due to (\ref{26}). Thus,
$$
\nu^c_r\geq \mbox{dim}\,N\left(\alpha_{S}-\langle\,,\,\rangle\gamma\right)\geq 
\mbox{dim}\,\left(\mbox{Ker}\,\Psi\right)= n-\mbox{dim}\,\left(\mbox{Im}\,\Psi\right)\geq n-r-1
$$
which is in contradiction with our hypothesis on the $r$-conformal nullity. Then $r=0$ and $\tilde{\nabla}^\perp_Z\mu_{d+1}=\left\langle \tilde{\nabla}^\perp_Z\mu_{d+1},\mu_{o}\right\rangle G$ for all $q\in U$ and $Z\in T_qM$ by (\ref{13}). From (\ref{25}) and $A^G_G=-I,$ we deduce that
$$
\left\langle \tilde{\nabla}^\perp_Z\mu_{d+1},\mu_{o}\right\rangle X=\left\langle \tilde{\nabla}^\perp_X\mu_{d+1},\mu_{o}\right\rangle Z,\;\forall q\in U,\;\forall X,Z\in T_qM.
$$
Since $n\geq 2,$ this implies that  $\left\langle \tilde{\nabla}^\perp_Z\mu_{d+1},\mu_{o}\right\rangle=0$ for all $q\in U$ and $Z\in T_qM$. Hence 
$\tilde{\nabla}^\perp_Z\mu_{d+1}=0$ and the unitary vector field $\mu_{d+1}$ is parallel along $U$. Then $L,$ the orthogonal complement of $\mbox{span}\,\{\mu_{d+1}\}$ in $T_G^\perp U,$ is a $(d+2)$-dimensional nondegenerate parallel subbundle of $T_G^\perp U,$ that is, $\tilde{\nabla}^\perp_X\mu\in L$  for all $X\in TU$ and $\mu\in L.$ We claim that $\mu_{d+1}$ is constant along $U$. In fact, for all $X\in TU$ we have
$$\tilde{\nabla}_X\mu_{d+1}=-A^G_{\mu_{d+1}}X+\tilde{\nabla}_X^\perp\mu_{d+1}=0.$$
We fix $q\in U$ and consider a differentiable curve $\gamma(t)$ in $U$ with $\gamma(0)=q.$ Then, putting $\mu_{d+1}(q)=\mu,$ we have
\begin{eqnarray*}
\frac{d}{dt}\langle G(\gamma(t))-G(q),\mu\rangle&=&\frac{d}{dt}\langle G(\gamma(t))-G(q),\mu_{d+1}(t)\rangle\\ 
&=&\left\langle \left(G_*\right)_{\gamma(t)}\gamma^\prime(t),\mu_{d+1}(t)\right\rangle +\left\langle G(\gamma(t))-G(q),\tilde{\nabla}_{\gamma^\prime(t)}\mu_{d+1}(t)\right\rangle =0.
\end{eqnarray*}
Therefore, $\langle G(\gamma(t))-G(q),\mu\rangle$ is constant and equal to zero. So $G(\gamma(t))$ lies on the $(n+d+2)$-dimensional  affine Lorentz space $G(q)+T_qM\oplus L(q).$ Since $\gamma (t)$ is arbitrary, it follows that $G(U)$ lies on $G(q)+T_qM\oplus L(q).$ Notice that $G(q)$ belongs to the vector space $L(q)$ and, consequently, $-G(q)$ also belongs to $L(q).$ Thus, $G(q)+T_qM\oplus L(q)$ pass through the origin. Hence, $G(U)\subset \left(G(q)+T_qM\oplus L(q)\right)\cap \mathcal{V}^{n+d+2}=\mathcal{V}^{n+d+1}.$ By (\ref{2}), $J_{\overline{\zeta}} (g(U))\subset H_{\overline{\zeta}}\cap \mathcal{V}^{n+d+1}=\Bbb R^{n+d}.$ This implies that $g$ restrict to $U$ reduces codimension to $n+d.$ So we can apply Theorem 1.2 in \cite{C-D} and Corollary 1.1 in \cite{Si} for conclude that there is a conformal diffeomorphism $\Gamma$ from an open subset of $\Bbb R^{n+d}$ containing $f(U)$ to an open subset of $\Bbb R^{n+d}$ containing $g(U)$ such that $g=\Gamma\circ f.$ This finish proof of Case II and of Theorem 1.2.\end{proof}


\begin{thebibliography}{999}

\bibitem{Ca} E. Cartan, La d\'eformation des hypersurfaces dans l'espace conforme r\'eel \`a $n\geq 5$ dimensions, Bull. Soc. Math. France. 45 (1917) 57--121. Jbuch. 46 p. 1129.

\bibitem{C-D} M. do Carmo and M. Dajczer, Conformal rigidity,
Amer. J. of Math. 109 (1987) 963--985.

\bibitem{Da} M. Dajczer, Submanifolds and isometric immersions, Math. Lec. Series 13, Publish or Perish Inc. Houston, 1990.

\bibitem{D-F} M. Dajczer and Luis A. Florit, A counterexample to a conjecture on flat bilinear forms, Proc. of the Amer. Math. Soc. 132 (2004) 3703--3704.

\bibitem{D-T} M. Dajczer and R. Tojeiro, A rigidity theorem for conformal immersions,
Indiana Univ. Math. J. 46 (1997) 491--504.

\bibitem{D-V} M. Dajczer and E. Vergasta, On composition of conformal immersions,
Proc. of the Amer. Math. Soc. 118 (1993) 211--215.

\bibitem{Si} Sergio L. Silva, On isometric and conformal rigidity of submanifolds,
Pacific J. of Math. 199 (2001) 177--247.

\end{thebibliography}
\end{document}